\documentclass{amsart}

\usepackage{amsfonts}
\usepackage{amsmath}
\usepackage{amssymb}
\usepackage{amsthm}

\makeatletter
\renewcommand\@seccntformat[1]{\csname the#1\endcsname.\quad}

\newtheorem{theorem}{Theorem}[section]

\newtheorem{corollary}[theorem]{Corollary}

\newtheorem{example}[theorem]{Example}

\newtheorem{lemma}[theorem]{Lemma}

\newtheorem{proposition}[theorem]{Proposition}
\newtheorem{remark}[theorem]{Remark}

\newcommand{\norm}[3]{\ensuremath{\left\Vert#1\right\Vert_{#2}^{#3}}}
\newcommand{\abs}[3]{\ensuremath{\left\vert#1\right\vert_{#2}^{#3}}}

\begin{document}

\author[M. Ciesielski]{Maciej Ciesielski}

\thanks{$^*$This research is supported by the grant 04/43/DSPB/0089 from Polish Ministry of Science and Higher Education.}

\title[$SKM$ and $KOC$ in symmetric spaces]{Strict $K$-monotonicity and $K$-order continuity in symmetric spaces}

\begin{abstract}
This paper is devoted to strict $K$- monotonicity and $K$-order continuity in symmetric spaces. Using the local approach to the geometric structure in a symmetric space $E$ we investigate a connection between strict $K$-monotonicity and global convergence in measure of a sequence of the maximal functions. Next, we solve an essential problem whether an existence of a point of $K$-order continuity in a symmetric space $E$ on $[0,\infty)$ implies that the embedding $E\hookrightarrow{L^1}[0,\infty)$ does not hold. We finish this article with a complete characterization of $K$-order continuity in a symmetric space $E$ that is written using a notion of order continuity under some assumptions on the fundamental function of $E$.  

\end{abstract}

\maketitle

%\bigskip\ 

{\small \underline{2000 Mathematics Subjects Classification: 46E30, 46B20, 46B42.}\hspace{1.5cm}\
\ \ \quad\ \quad  }\smallskip\ 

{\small \underline{Key Words and Phrases:}\hspace{0.15in} Symmetric space, Marcinkiewicz space, $K$-order continuity, strict $K$-monotonicity, the Kadec-Klee property for global convergence in measure.}

%\bigskip\ \ 

\section{Introduction}

In 1992, Kurc \cite{Kurc} presented a relation between the best dominated approximation and order continuity and monotonicity properties. In the paper \cite{CieKolPan}, authors answered the question about the full characterization of the local monotonicity structure and order continuity of Banach lattices and applications in the best dominated approximation. The next motivating research was published in \cite{CieKolPlu}, where there has been established among others a connection between the best dominated approximation and the Kadec-Klee property for global convergence in measure in Banach function spaces. Recently, in view of the previous investigation there appeared many results \cite{Cies-geom,CKKP,CiesKamPluc,CieKolPluSKM,HKL} devoted to exploration of the global and local monotonicity and rotundity structure of Banach spaces applicable in the best approximation problems. The main inspiration for this article appeared in paper \cite{Cies-JAT}, where there has been introduced a new type of the best dominated approximation with respect to the Hardy-Littlewood-P\'olya relation $\prec$. It is worth mentioning that author has investigated an application of strict  $K$-monotonicity and $K$-order continuity to the best dominated approximation problems in the sense of the Hardy-Littlewood-P\'olya relation. In the spirit of the previous articles, the essential purpose of this paper is to focus on the complete criteria for $K$-order continuity in symmetric spaces. 

It is necessary to recall the significant paper \cite{ChDSS}, in which there has been shown a correlation between strict $K$-monotonicity and global convergence in measure of a sequence of the decreasing rearrangements in symmetric spaces. The next intention of this paper is to find a local version of a correspondence between strict $K$-monotonicity and global convergence in measure of a sequence of the maximal functions in symmetric spaces.  

The present article is organized as follows. Preliminaries contain all needed terminologies, which will be used later. Here we also recall an earlier result, which play a crucial role in our further exploration. 
 
Section 3 is dedicated to strict $K$-monotonicity in a symmetric space $E$. We start our investigation with the auxiliary result proving an existence of at most two upper bounds of $\mathcal{M}(x,\tau,\epsilon)$ the family of the decreasing rearrangements bounded by an element $x\in{E}$ with respect to the Hardy-Littlewood-P\'olya relation $\prec$. Then, in view of the above result, we present an application of a point of lower $K$-monotonicity to global convergence in measure of a sequence from the cone of all nonnegative decreasing elements in $E$. 

In the last section 4 we answer the key question whether an existence of a point of $K$-order continuity in a symmetric space $E$ on $[0,\infty)$ guarantees that $E$ is not embedded in $L^1[0,\infty)$. Next, we discuss the full criteria for a point of $K$-order continuity in a symmetric space $E$. Namely, we show a complete correlation between a point of $K$-order continuity and a point of order continuity in a symmetric space $E$ under an additional assumption of the fundamental function $\phi$ in case when $I=[0,\infty)$. Finally, we describe a relationship between a point of order continuity and an $H_g$ point in a symmetric space $E$. 

\section{Preliminaries}

Assume that $\mathbb{R}$, $\mathbb{R}^+$ and $\mathbb{N}$ are the sets of reals, nonnegative reals and positive integers, respectively.  For a Banach space $(X,\norm{\cdot}{X}{})$ we denote by $S(X)$ (resp. 
$B(X))$ the unit sphere (resp. the closed unit ball). A nonnegative function $\phi$ defined on $\mathbb{R}^+$ is said to be \textit{quasiconcave} if $\phi(t)$ is increasing and $\phi(t)/t$ is decreasing on $\mathbb{R}^+$ and also $\phi(t)=0\Leftrightarrow{t=0}$. Let us denote as usual by $\mu$ the Lebesgue measure on $I=[0,\alpha)$, where $\alpha =1$ or $\alpha =\infty$, and by $L^{0}$ the set of all (equivalence classes of) extended real valued Lebesgue measurable functions on $I$. We employ the notation $A^{c}=I\backslash A$ for any measurable set $A$. Given a Banach lattice $(E,\Vert \cdot \Vert _{E})$ is said to be a \textit{Banach
function space} (or a \textit{K\"othe space}) if it is a sublattice of $L^{0}
$ satisfying the following conditions
\begin{itemize}
\item[(1)] If $x\in L^0$, $y\in E$ and $|x|\leq|y|$ a.e., then $x\in E$ and $%
\|x\|_E\leq\|y\|_E$.
\item[(2)] There exists a strictly positive $x\in E$.
\end{itemize}
Additionally, to simplify the notation in the whole paper it is used the symbol $E^{+}={\{x \in E:x \ge 0\}}$. 

An element $x\in E$ is called a \textit{point of order continuity} if for any
sequence $(x_{n})\subset{}E^+$ such that $x_{n}\leq \left\vert x\right\vert 
$ and $x_{n}\rightarrow 0$ a.e. we have $\left\Vert x_{n}\right\Vert
_{E}\rightarrow 0.$ A Banach function space $E$ is said to be \textit{order continuous 
}(shortly $E\in \left( OC\right) $) if any element $x\in{}E$ is a point of order continuity (see \cite{LinTza}). We say that a Banach function space $E$ has the \textit{Fatou property} if for any $\left( x_{n}\right)\subset{}E^+$, $\sup_{n\in \mathbb{N}}\Vert x_{n}\Vert
_{E}<\infty$ and $x_{n}\uparrow x\in L^{0}$, then we get $x\in E$ and $\Vert x_{n}\Vert _{E}\uparrow\Vert x\Vert
_{E}$. Unless it is said otherwise, in the whole article it is considered that $E$ has the Fatou property.

% $Kadec-Klee$ porperties
An element $x\in{E}$ is called an $H_g$ \textit{point} in $E$ if for any sequence $(x_n)\subset{E}$ such that $x_n\rightarrow{x}$ globally in measure and $\norm{x_n}{E}{}\rightarrow\norm{x}{E}{}$, then $\left\Vert x_n-x\right\Vert _{E}\rightarrow{0}$. Let us recall that the space $E$ has the \textit{Kadec-Klee property for global convergence in measure} if each element $x\in{E}$ is an $H_g$ point in $E$. 

The \textit{distribution function} for any function $x\in L^{0}$ is given by 
\begin{equation*}
d_{x}(\lambda) =\mu\left\{ s\in [ 0,\alpha) :\left\vert x\left(s\right) \right\vert >\lambda \right\},\qquad\lambda \geq 0.
\end{equation*}
For each element $x\in L^{0}$ we define its \textit{decreasing rearrangement} by 
\begin{equation*}
x^{\ast }\left( t\right) =\inf \left\{ \lambda >0:d_{x}\left( \lambda
\right) \leq t\right\}, \text{ \ \ } t\geq 0.
\end{equation*}
 We use the convention $x^{*}(\infty)=\lim_{t\rightarrow\infty}x^{*}(t)$ if $\alpha=\infty$ and $x^*(\infty)=0$ if $\alpha=1$. For any function $x\in L^{0}$ we define the \textit{maximal function} of $x^{\ast }$ by 
\begin{equation*}
x^{\ast \ast }(t)=\frac{1}{t}\int_{0}^{t}x^{\ast }(s)ds.
\end{equation*}
Given $x\in L^{0}$, it is commonly known that $x^{\ast }\leq x^{\ast \ast },$ $x^{\ast \ast }$ is decreasing, continuous and subadditive. For more details of $d_{x}$, $x^{\ast }$ and $x^{\ast \ast }$ the reader is referred to see \cite{BS, KPS}. 

Let us recall that two functions $x,y\in{L^0}$ are called \textit{equimeasurable} (shortly $x\sim y$) if $d_x=d_y$. A Banach function space $(E,\Vert \cdot \Vert_{E}) $ is said to be \textit{symmetric} or \textit{rearrangement invariant} (r.i. for short) if whenever $x\in L^{0}$ and $y\in E$ such that $x \sim y,$ then $x\in E$ and $\Vert x\Vert_{E}=\Vert y\Vert _{E}$. For a symmetric space $E$ we define $\phi_{E}$ its \textit{fundamental function} given by $\phi_{E}(t)=\Vert\chi_{(0,t)}\Vert_{E}$ for any $t\in [0,\alpha)$ (see \cite{BS}). Given $x,y\in{}L^{1}+L^{\infty }$ we define the \textit{Hardy-Littlewood-P\'olya relation} $\prec$ by 
\begin{equation*}
x\prec y\Leftrightarrow x^{\ast \ast }(t)\leq y^{\ast \ast }(t)\text{ for
all }t>0.\text{ }
\end{equation*}

% $K$-monotonicity
A symmetric space $E$ is said to be $K$-\textit{monotone} (shortly $E\in(KM)$)
if for any $x\in L^{1}+L^{\infty}$ and $y\in E$ with $x\prec y,$ then $x\in E$ and $\Vert x\Vert_{E}\leq \Vert y\Vert _{E}.$ It is well known that a symmetric space is $K$-monotone if and only if $E$ is exact
interpolation space between $L^{1}$ and $L^{\infty }.$ Let us also recall that a symmetric space $E$ equipped with an order continuous norm or with the Fatou property is $K$-monotone (see \cite{KPS}).

% strict $K$-monotonicity
Given $x\in{E}$ is said to be a \textit{point of lower $K$-monotonicity} of $E$ (shortly an $LKM$ \textit{point}) if for any $y\in{E}$, $x^*\neq{y^*}$ with $y\prec{x}$, then $\norm{y}{E}{}<\norm{x}{E}{}$. Let us mention that given a symmetric space $E$ is called \textit{strictly $K$-monotone} (shortly $E\in(SKM)$) if any element of $E$ is an $LKM$ point. 

% $K$ order continuity
An element $x\in{E}$ is called a \textit{point of $K$-order continuity} of $E$ if for any sequence $(x_n)\subset{E}$ with $x_n\prec{x}$ and $x_n^*\rightarrow{0}$ a.e. we have $\norm{x_n}{E}{}\rightarrow{0}$. A symmetric space $E$ is said to be $K$-\textit{order continuous} (shortly $E\in\left(KOC\right)$) if any element $x$ of $E$ is a point of $K$-order continuity. We refer the reader for more information to see \cite{ChDSS,Cies-JAT,Cies-geom,CieKolPluSKM}.

% Marcinkiewicz space $M_{\phi}$
For any quasiconcave function $\phi$ on $I$ the Marcinkiewicz function space $M_{\phi}^{(*)}$ (resp. $M_{\phi}$) is a subspace of $L^0$ such that
$$\norm{x}{M_{\phi}^{(*)}}{}=\sup_{t>0}\{x^*(t)\phi(t)\}<\infty$$
$${\left(\textnormal{resp. }\norm{x}{M_{\phi}}{}=\sup_{t>0}\{x^{**}(t)\phi(t)\}<\infty\right).}$$
Clearly, $M_{\phi}\overset{1}{\hookrightarrow}M_{\phi}^{(*)}$. It is well known that the Marcinkiewicz space $M_{\phi}^{(*)}$ (resp. $M_{\phi}$) is a r.i. quasi-Banach function space (r.i. Banach function space) with the fundamental function $\phi$ on $I$. It is worth reminding that for any symmetric space $E$ with the fundamental function $\phi$ we have $E\overset{1}{\hookrightarrow}M_{\phi}$ (for more details see \cite{BS,KPS}). 

% auxiliary results
Let us recall auxiliary result proved in \cite{Cies-JAT} which will be useful in our further investigation.
\begin{lemma}\label{lem:x**=0}
	Let $x\in{L^1+L^\infty}$ and $x^*(\infty)=0$, then $x^{**}(\infty)=0$. 
\end{lemma}

\section{local approach to strict $K$-monotonicity in symmetric spaces}

Now, we introduce the special set which is similar to a notion that was presented in \cite{KPS}. Let $E$ be a symmetric space and $x\in{E}$, $x=x^*$. Define for any $\epsilon>0$ and $\tau>0$,
\begin{equation*}
\mathcal{M}(x,\tau,\epsilon)=\left\{y\in{E}:y=y^*,y\prec{x},\int_{0}^{\tau}y(s)ds+\epsilon\leq\int_{0}^{\tau}x(s)ds\right\}.
\end{equation*}
The next result shows that for any collection $\mathcal{M}(x,\tau,\epsilon)$ there exist at most two elements that are upper bounds of a given family less than $x$ with respect to the relation $\prec$. The proof of the below result is quite long and required completely different techniques than the proof stated in Corollary 2.2 \cite{ChDSS}. Although the case $1$ in the proof follows from the case $2$, we present the whole details for the sake of the reader's convenience and because of the construction that might arouse the reader's interest.

\begin{proposition}\label{prop:2:upper:bound}
Let $x\in{L^0}$, $x=x^*$, $x^*(\infty)=0$ and $\tau>0$, $\epsilon>0$. Then, there exist $\epsilon_1>0$, $\tau_1\in(0,\tau)$, $z\in\mathcal{M}(x,\tau-\tau_1,\epsilon_1)$ and $w\in\mathcal{M}(x,\tau+\tau_1,\epsilon_1)$ such that for any $y\in\mathcal{M}(x,\tau,\epsilon)$ we have $y\prec{z}$ or $y\prec{w}$.
\end{proposition}

\begin{proof}
 Let $y\in\mathcal{M}(x,\tau,\epsilon)$. Define for any decreasing function $u\in{L^0}$  and $t>0$,
\begin{equation*}
\phi_u(t)=\int_{0}^{t}u(s)ds.
\end{equation*}
Since $x^*(\infty)=0$, by monotonicity of functions $x$ and $y$ it is clear that $\phi_x$ and $\phi_y$ are concave, increasing and not affine on $I$. By assumption we have $\phi_y\leq\phi_x$ on $I$ and also $\phi_y(\tau)\leq\phi_x(\tau)-\epsilon$. Let $\gamma$ be an intersection of $\phi_x$ and $p(t)=\phi_x(\tau)-\epsilon$ and let $\beta$ be an intersection of $\phi_x$ and $q(t)=t(\phi_x(\tau)-\epsilon)/\tau$. Denote
\begin{equation*}
\xi=\frac{\phi_x(\beta)-\phi_x(\gamma)}{\beta-\gamma}=\frac{1}{\beta-\gamma}\int_{\gamma}^{\beta}x^*\quad\textnormal{and}\quad{\phi(t)=\xi(t-\gamma)+\phi_x(\gamma)}.
\end{equation*}
Now, we show the proof in cases.\\
\textit{Case $1.$} Assume that there exists $t\in(\gamma,\beta)$ such that $\phi(t)<\phi_x(t)$. Then, by concavity of $\phi_x$ for any $t\in(\gamma,\beta)$ we have $\phi(t)<\phi_x(t)$. Define 
\begin{equation*}
z=x^*\chi_{[0,\gamma)\cup[\beta,\infty)}+\xi\chi_{[\gamma,\beta)}.
\end{equation*}
Since $\xi$ is an average value of $x^*$ on $(\gamma,\beta)$, by monotonicity of $x^*$ we observe $z=z^*$, $z\prec{x}$ and $z\neq{x}$. Moreover, by concavity of $\phi_y$ we notice that a tangent line of $\phi_y$ at $\tau$ has a slope $\eta\in[0,(\phi(\tau)-\epsilon)/\tau]$ and for any $t\in{I}$,
\begin{equation*}
\phi_y(t)\leq{\eta(t-\tau)+\phi_x(\tau)-\epsilon}.
\end{equation*}
Hence, for any $t\in[\gamma,\beta]$ we have
\begin{equation*}
\phi_y(t)\leq{\eta(t-\tau)+\phi_x(\tau)-\epsilon}\leq{\xi(t-\tau)+\phi_x(\gamma)}=\phi(t)=\phi_z(t).
\end{equation*}
Consequently, since $\phi_x=\phi_z$ on $[0,\gamma]\cup{[\beta,\infty)}$, by assumption $y\prec{x}$ we get $y\prec{z}$.\\
\textit{Case $2$.} Assume that for any $t\in[\gamma,\beta]$, $\phi_x(t)=\phi(t)$. Denote
\begin{equation*}
\gamma_0=\inf\{t>0:\phi_x(t)=\phi(t)\}.
\end{equation*}
Clearly, $\gamma_0\leq{\gamma}$. We claim that $\gamma_0>0$. Indeed, if we suppose that $\gamma_0=0$, then $\phi_x=\phi$ on $[0,\beta]$ and so $\phi_x(0)=-\xi\gamma+\phi_x(\gamma)=0$. Hence, by definition of $\xi$ and by assumption that $\beta$ is the intersection of $\phi_x$ and $q$ we have $$\xi=\frac{\phi_x(\gamma)}{\gamma}=\frac{\phi_x(\beta)}{\beta}=\frac{\phi_x(\tau)-\epsilon}{\tau}.$$ 
On the other hand, since $\phi_x$ is affine on $[\gamma_0,\beta]$ it follows that $\xi=\phi_x(\tau)/\tau$. Therefore, we obtain a contradiction which proves our claim. Now, without loss of generality we may suppose that $\epsilon>0$ is small enough such that there exist $\gamma_1>0$ and $\beta_1>0$ the intersections of the functions $\phi_x$ and $\phi-\epsilon$. Observe $0<\gamma_1<\gamma_0<\gamma<\beta<\beta_1$. Consider temporary the slope $\eta$ of the tangent line of $\phi_y$ at $\tau$ is $\eta\in[\xi,(\phi_x(\tau)-\epsilon)/\tau]$. Define
\begin{equation*}
z=x^*\chi_{[0,\gamma_1)\cup[\tau,\infty)}+\frac{\phi_x(\tau)-\phi_x(\gamma_1)}{\tau-\gamma_1}\chi_{[\gamma_1,\tau)}.
\end{equation*}
Then, since $\frac{\phi_x(\tau)-\phi_x(\gamma_1)}{\tau-\gamma_1}$ is the average value of $x^*$ on $(\gamma_1,\tau)$, by definition of $\gamma_0$ and by assumption $\phi_x$ is affine on the interval $(\gamma_0,\beta)$ and by the concavity of $\phi_x$ it follows that $z=z^*$, $z\prec{x}$ and $z\neq{x}$. Furthermore, we have for every $t\in[\gamma_1,\tau]$, 
\begin{equation*}
\phi_y(t)\leq{\eta(t-\tau)+\phi_x(\tau)-\epsilon}\leq\phi_z(t).
\end{equation*}
Thus, since $\phi_x=\phi_z$ on $[0,\gamma_1]\cup[\tau,\infty)$, by assumption $y\prec{x}$ we conclude $y\prec{z}$, whenever $\eta\in[\xi,(\phi_x(\tau)-\epsilon)/\tau]$. Now, assume that $\eta\in[0,\xi)$. Define
\begin{equation*}
w=x^*\chi_{[0,\gamma)\cup[\beta_1,\infty)}+\frac{\phi_x(\beta_1)-\phi_x(\gamma)}{\beta_1-\gamma}\chi_{[\gamma,\beta_1)}.
\end{equation*}
Since $\frac{\phi_x(\beta_1)-\phi_x(\gamma)}{\beta_1-\gamma}$ is the average value of $x^*$ on $[\gamma,\beta_1]$, by definition of $\beta_1$ and by concavity of $\phi_x$ we get $w=w^*$, $w\prec{x}$ and $w\neq{x}$. Notice that for any $t\in[\gamma,\beta_1]$,
\begin{equation*}
\phi_y(t)\leq{\eta(t-\tau)+\phi_x(\tau)-\epsilon}\leq\phi_w(t).
\end{equation*}
Moreover, $\phi_x=\phi_w$ on $[0,\gamma]\cup[\beta_1,\infty)$ and by assumption $y\prec{x}$ we obtain $y\prec{w}$ for any $\eta\in[0,\xi)$ which completes case $2$. \\
Finally, combining both cases, by construction of $z$ and $w$ it is easy to see that there exist $\tau_1\in(0,\tau)$ and $\epsilon_1>0$ such that $z\in\mathcal{M}(x,\tau-\tau_1,\epsilon_1)$ and $w\in\mathcal{M}(x,\tau+\tau_1,\epsilon_1)$.
\end{proof}

\begin{theorem}\label{thm:conv:in:meas}
Let $E$ be a symmetric space. Assume that $x\in{E}$ is an $LKM$ point, $x^*(\infty)=0$ and $(x_n)\subset{E}$. If $x_n\prec{x}$ and $\norm{x_n}{E}{}\rightarrow\norm{x}{E}{}$, then $x_n^{**}\rightarrow{x}^{**}$ and $x_n^*\rightarrow{x^*}$ globally in measure.
\end{theorem}

\begin{proof}
Let $\phi$ be the fundamental function of $E$. Since $x_n\prec{x}$ for any $n\in\mathbb{N}$, by symmetry and the Fatou property of $E$ and by Theorem 4.6 in \cite{BS} it follows that 
\begin{equation*}
x_n^*(t)\phi(t)\leq\norm{x_n}{E}{}\leq\norm{x}{E}{}
\end{equation*}
for any $t\in{I}$ and $n\in\mathbb{N}$. Hence, by Helly's Selection Theorem in \cite{STV} passing to subsequence and relabelling if necessary there is $z\in{E}$ such that $z=z^*$ and $x_n^*$ converges to $z^*$ a.e. on $I$. It is easy to notice that for all $0<t<s<\alpha$,
\begin{equation}\label{equ:1:cim}
\int_{t}^{s}x_n^*\rightarrow\int_{t}^{s}z^*.
\end{equation}
We claim that for any $t>0$,
\begin{equation}\label{equ:2:cim}
\int_0^{t}x_n^*\rightarrow\int_0^{t}x^*.
\end{equation}
Indeed, if it is not true then there exist $\tau>0$ and $\epsilon>0$ and a sequence $(n_k)\subset\mathbb{N}$ such that for all $k\in\mathbb{N}$,
\begin{equation*} %\label{equ:1:cim}
\int_0^{\tau}x_{n_k}^*\leq\int_0^{\tau}x^*-\epsilon.
\end{equation*}
Consequently, by Proposition \ref{prop:2:upper:bound} there are $\epsilon_1>0$, $\tau_1>0$ and $y\in\mathcal{M}(x,\tau-\tau_1,\epsilon_1)$ and $w\in\mathcal{M}(x,\tau+\tau_1,\epsilon_1)$ such that for any $k\in\mathbb{N}$,
\begin{equation*}
x_{n_k}\prec{y}\quad\textnormal{or}\quad{x_{n_k}\prec{w}}.
\end{equation*}
Since $y^*\neq{x}^*$, $w^*\neq{x^*}$ and $y\prec{x}$, $w\prec{x}$, by assumption that $x$ is an $LKM$ point we get for every $k\in\mathbb{N}$, $$\norm{x_{n_k}}{E}{}\leq\max\{\norm{y}{E}{},\norm{w}{E}{}\}<\norm{x}{E}{}.$$
Thus, by assumption $\norm{x_{n_k}}{E}{}\rightarrow\norm{x}{E}{}$ we obtain a contradiction which provides our claim \eqref{equ:2:cim}. Now, combining conditions \eqref{equ:1:cim} and \eqref{equ:2:cim} we have $x^*=z^*$ a.e. on $I$. Therefore, $x_n^*$ converges to $x^*$ a.e. on $I$. Since $x^*(\infty)=0$, it is easy to prove that $x_n^*$ converges to $x^*$ globally in measure. Indeed, there exists $t_0>0$ such that $x^*(t)<\epsilon/2$ for any $t\geq{t_0}$. Without loss of generality we may assume that $$|x_n^*(t_0)-x^*(t_0)|<\epsilon/2$$
for large enough $n\in\mathbb{N}$. In consequence, by monotonicity of a decreasing rearrangement for all $t\geq{t_0}$ and large enough $n\in\mathbb{N}$ we obtain
\begin{equation*}
x_n^*(t)\leq{x_n^*(t_0)}\leq\abs{x_n^*(t_0)-x^*(t_0)}{}{}+x^*(t_0)<\epsilon.
\end{equation*}
Hence, for any $t\geq{t_0}$ and for large enough $n\in\mathbb{N}$ we get 
\begin{equation*}
\abs{x_n^*(t)-x^*(t)}{}{}<\epsilon
\end{equation*}
and since $x_n^*$ converges to $x^*$ locally in measure we conclude global convergence in measure. Moreover, by Lemma \ref{lem:x**=0} we have $x^{**}(\infty)=0$. Hence, there is $t_0>0$ such that for all $t\geq{t_0}$ we have $x^{**}(t)<{\epsilon}$. Since $x_n\prec{x}$ for any $n\in\mathbb{N}$ and by condition \eqref{equ:2:cim} it follows that $x_n^{**}$ converges to $x^{**}$ globally in measure.
\end{proof}

\section{$K$-order continuity in symmetric spaces}

In this section we study a full criteria for $K$-order continuity in symmetric spaces. We start our investigation with an equivalent condition for $K$-order continuity in a symmetric space $E$. This notion was discussed in paper \cite{DSS} for a separable symmetric space with the Fatou property. We study $K$-order continuity in more general case and present a simple example of the order continuous rearrangement invariant space with the Fatou property that is not $K$-order continuous, showing that Proposition 4.2 in \cite{DSS} fails in general case.

\begin{lemma}
Let $E$ be a symmetric space and $x_n,x\in{E}$ for all $n\in\mathbb{N}$ and let: 
\begin{itemize}
	\item[$(i)$] If $x_n\prec{x}$ and $x_n^*\rightarrow{0}$ a.e., then $\norm{x_n}{E}{}\rightarrow{0}.$
	\item[$(ii)$] If $x_n\prec{x}$ and $x_n\rightarrow{0}$ globally in measure, then $\norm{x_n}{E}{}\rightarrow{0}.$
\end{itemize}
Then, $(i)\Rightarrow(ii)$. If $x^*(\infty)=0$ then $(ii)\Rightarrow(i)$. 
\end{lemma}

\begin{proof}
$(i)\Rightarrow(ii)$. Let assumption be satisfied. Then, by property 2.11 in \cite{KPS} it follows that $x_n^*\rightarrow{0}$ a.e. Consequently, since $x_n\prec{x}$ for any $n\in\mathbb{N}$, by symmetry of $E$ and by condition $(i)$ we get $\norm{x_n}{E}{}\rightarrow{0}$.\\
$(ii)\Rightarrow(i)$. Let $x_n\prec{x}$ for all $n\in\mathbb{N}$ and $x_n^*\rightarrow{0}$ a.e. and let $x^*(\infty)=0$. Then, by Lemma 3.1 in \cite{Cies-JAT} we get $x^{**}(\infty)=0$. Hence, since 
\begin{equation*}
x_n^*(t)\leq{x_n^{**}(t)}\leq{x^{**}(t)}
\end{equation*}
for any $t>0$ and $n\in\mathbb{N}$, we easily observe that $x_n^*\rightarrow{0}$ globally in measure. Therefore, since $x_n^*\prec{x}$ for any $n\in\mathbb{N}$, by symmetry of $E$ and by condition $(ii)$ we complete the proof.
\end{proof}

Immediately, by definition of $K$-order continuity and by the above lemma as well as by Lemma 2.5 in \cite{CieKolPan} in case when $E$ is order continuous we conclude the following corollary.

\begin{corollary}
	Let $E$ be a symmetric space and $x_n,x\in{E}$ for all $n\in\mathbb{N}$. If for any $x\in{E}$ we have $x^*(\infty)=0$ or $E$ is order continuous then the following condition are equivalent.
	\begin{itemize}
	\item[$(i)$] $E$ is $K$-order continuous.
    \item[$(ii)$] If $x_n\prec{x}$ and $x_n\rightarrow{0}$ globally in measure, then $\norm{x_n}{E}{}\rightarrow{0}.$		
	\end{itemize}
\end{corollary} 

Now we answer the question about a correspondence between a point of $K$-order continuity in a symmetric space $E$ and an embedding $E\hookrightarrow{L^1[0,\infty)}$.

\begin{lemma}\label{lem:2:KOC}
Let $E$ be a symmetric space on $I=[0,\infty)$ and let $\phi$ be the fundamental function of $E$. If $x\in{E}\setminus\{0\}$ is a point of $K$-order continuity, then $\lim_{t\rightarrow\infty}{\phi(t)}/{t}=0.$ Additionally, if $x^*(\infty)=0$ we have $\lim_{t\rightarrow\infty}\phi(t)x^{**}(t)=0.$
\end{lemma}

\begin{proof}
Let $t_x>0$ be such that $x^*(t_x)>0$. Define $$y_n=\frac{x^*(t_x)}{n}\chi_{[0,nt_x)}$$ for any $n\in\mathbb{N}$. Clearly,  $y_n=y_n^*\prec{}y_1^*\leq{x^*}$ and $\int_{0}^{\infty}y_n^*=x^*(t_x)t_x<\infty$ for any $n\in\mathbb{N}$. Thus, since $y_n^*\rightarrow{0}$ a.e. and by assumption that $x$ is a point of $K$-order continuity we obtain 
\begin{equation*}
\norm{y_n}{E}{}={x^*(t_x)}\frac{\phi(nt_x)}{n}={x^*(t_x)t_x}\frac{\phi(nt_x)}{nt_x}\rightarrow{0}\quad\textnormal{as}\quad{n\rightarrow\infty}.
\end{equation*}
Hence, $\phi(s)/s\rightarrow{0}$ as $s\rightarrow\infty$.
Next, assume that $x^*(\infty)=0$. Denote for any $n\in\mathbb{N}$,
\begin{equation*}
x_n=x^{**}(n)\chi_{[0,n)}.
\end{equation*}	
Then, we have $x_n=x_n^*\prec{x}$ for all $n\in\mathbb{N}$. Since $x^*(\infty)=0$, by Lemma \ref{lem:x**=0} it follows that $x_n^*\rightarrow{0}$ a.e. on $I$. Finally, according to assumption that $x$ is a point of $K$-order continuity we get
$\norm{x_n}{E}{}=x^{**}(n)\phi(n)\rightarrow{0}$ as $n\rightarrow\infty$.
\end{proof}

\begin{theorem}
Let $E$ be a symmetric space on $I=[0,\infty)$. If a point $x\in{E}$ is a point of $K$-order continuity, then $E$ is not embedded in $L^1[0,\infty)$.
\end{theorem}

\begin{proof}
Let $\phi$ be a fundamental function of a symmetric space $E$ on $I=[0,\infty)$. Then, we have $\lim_{t\rightarrow\infty}\frac{\phi(t)}{t}>0$ if and only if $\phi(t)\approx{t}$ for large enough $t>0$. Let us mention that the fundamental function $\phi$ of a symmetric space $E$ on $I=[0,\infty)$ satisfies $\phi(t)\approx{t}$ for large enough $t>0$ if and only if $E\overset{C}{\hookrightarrow}{L^1[0,\infty)}$. Hence, by assumption that $x$ is a point of $K$-order continuity, applying Lemma \ref{lem:2:KOC} we finish the proof.
\end{proof}

In the next results we investigate a connection between order continuity and $K$-order continuity in symmetric spaces.

\begin{lemma}\label{lem:KOC=>OC}
	Let $E$ be a symmetric space. If $x\in{E}$ is a point of $K$-order continuity and $x^*(\infty)=0$, then $x$ is a point of order continuity.
\end{lemma}

\begin{proof}
	Let $(x_n)\subset{E}^+$ with $x_n\rightarrow{0}$ a.e. and $x_n\leq\abs{x}{}{}$. Immediately, by the property of the maximal function we obtain $x_n\prec{x}$. Since $x^*(\infty)=0$ and $x_n\leq\abs{x}{}{}$ for all $n\in\mathbb{N}$, by property $2.12$ in \cite{KPS} it follows that $x_n^*\rightarrow{0}$ a.e. Hence, by assumption that $x$ is a point of $K$-order continuity, we have $\norm{x_n}{E}{}\rightarrow{0}$.
\end{proof}

We show that the reverse conclusion of Lemma \ref{lem:KOC=>OC} in general is not true. Namely, we prove that any element in $L^1[0,\infty)$ is not point of $K$-order continuity. Moreover, the following remark yields that Proposition 4.2 in \cite{DSS} is not true in case when $I=[0,\infty)$. 

\begin{remark} 
	It is well known that each element $x$ of $L^1[0,\infty)$ is a point of order continuity (see \cite{LinTza}). Consider $x=\chi_{[0,1)}$ and $x_n=\frac{1}{n}\chi_{[0,n)}$ for any $n\in\mathbb{N}$. It is easy to see that $x_n\rightarrow{0}$ and $x_n\prec{x}$ for every $n\in\mathbb{N}$. On the other hand, we can easily show $\norm{x_n}{L^1}{}=1$ for any $n\in\mathbb{N}$, which concludes that $x$ is no point of $K$-order continuity. Hence, it is easy to observe that any $y\in{L^1}\setminus\{0\}$ is no point of $K$-order continuity. Indeed, finding $\beta, \gamma\in\mathbb{R}^+$ such that $\beta\chi_{(0,\gamma)}\prec{y^*}$ and proceeding analogously as above with $x$ we finish our investigation.
\end{remark}

In the next example, we show that the assumption $x^*(\infty)=0$ in Lemma \ref{lem:KOC=>OC} is necessary and cannot be omitted.

\begin{example}
We consider the Marcinkiewicz function space $M_{\phi}^{(*)}$ with $\phi(t)=1-\frac{1}{1+t}$ on $I=[0,\infty)$. Define $x=\chi_{[0,\infty)}$. We claim that $x$ is a point of $K$-order continuity and it is no point of order continuity in $M_{\phi}^{(*)}$. It is easy to see that $x^*(\infty)=1$ and $\norm{x}{M_{\phi}^{(*)}}{}=1$. By Lemma 2.5 \cite{CieKolPan} it follows that $x$ is no point of order continuity in $M_{\phi}^{(*)}$. Let $(x_n)\subset{M_{\phi}^{(*)}}$ be such that $x_n\prec{x}$ and $x_n^*\rightarrow{0}$ a.e. In view of the fact that $\norm{\cdot}{M_{\phi}^{(*)}}{}$ satisfies the triangle inequality with a constant $C>0$, we observe for all $n\in\mathbb{N}$,
\begin{align}\label{equ:1:rem:2}
\norm{x_n}{M_{\phi}^{(*)}}{}&\leq{}C\norm{x_n^*\chi_{[0,1]}}{M_{\phi}^{(*)}}{}+C\norm{x_n^*\chi_{(1,\infty)}}{M_{\phi}^{(*)}}{}\\
&\leq{C}\norm{x_n^*\chi_{[0,1]}}{M_{\phi}^{(*)}}{}+{C}x_n^*(1)\norm{x}{M_{\phi}^{(*)}}{}\nonumber.
\end{align}
By assumption $x_n\prec{x}$ and by definition of $x$ and right-continuity of $x^*$ we get $x_n^*(t)\leq{1}$ for all $n\in\mathbb{N}$ and $t\in[0,\infty)$. Thus, by Bolzano-Weierstrass theorem (see \cite{Royd}), passing to subsequence and relabelling if necessary we may assume that for each $n\in\mathbb{N}$ there exists $t_n\in[0,1]$ such that $\lim_{n\rightarrow\infty}t_n=t_0$ and
\begin{equation}\label{equ:2:rem:2}
\norm{x_n^*\chi_{[0,1]}}{M_{\phi}^{(*)}}{}=\sup_{0<t<1}\left\{x_n^*(t)\phi(t)\right\}\leq{}x_n^*(t_n)\left(1-\frac{1}{1+t_n}\right)+\frac{1}{n}
\end{equation}
for all $n\in\mathbb{N}$. Since $x_n^*$ is decreasing for any $n\in\mathbb{N}$, if $t_0>0$, then it is easy to notice that for large enough $n\in\mathbb{N}$,
\begin{equation*}
x_n^*(t_n)\left(1-\frac{1}{1+t_n}\right)\leq{}x_n^*\left(\frac{t_0}{2}\right)\left(1-\frac{1}{1+t_n}\right)\leq{}x_n^*\left(\frac{t_0}{2}\right),
\end{equation*}
whence, by assumption $x_n^*\rightarrow{0}$ a.e. and by conditions \eqref{equ:1:rem:2} and \eqref{equ:2:rem:2}, this yields that $\norm{x_n}{M_{\phi}^{(*)}}{}\rightarrow{0}$. Now, suppose $t_0=0$. Then, $t_n\rightarrow{0}$ and since $x_n^*\leq{1}$ for all $n\in\mathbb{N}$ we have
\begin{equation*}
x_n^*(t_n)\left(1-\frac{1}{1+t_n}\right)\leq{}\left(1-\frac{1}{1+t_n}\right)\rightarrow{0}.
\end{equation*}
Hence, proceeding analogously as in previous case we complete the proof.
\end{example}

In the next results we investigate a complete characterization of $K$-order continuity in symmetric spaces.

\begin{theorem}
	Let $E$ be a symmetric space on $I=[0,\alpha)$, where $\alpha=1$ or $\alpha=\infty$, with the fundamental function $\phi$ and let $x\in{E}$. Then, the following conditions are equivalent.
	\begin{itemize}
		\item[$(i)$]  $x$ is a point of order continuity and in case when $\alpha=\infty$, $$\phi(s)x^{**}(s)\rightarrow{0}\quad\textnormal{as}\quad{}s\rightarrow\infty.$$
		\item[$(ii)$] $x$ is a point of $K$-order continuity and in case when $\alpha =\infty$, $$x^*(\infty)=0.$$
	\end{itemize} 
\end{theorem}

\begin{proof}
$(ii)\Rightarrow(i)$. It follows immediately from Lemmas \ref{lem:2:KOC} and \ref{lem:KOC=>OC}.\\
$(i)\Rightarrow(ii)$. We prove only the theorem in case when $\alpha=\infty$, because the proof in the case when $\alpha=1$ is similar.
Let $(x_n)\subset{E}$ and $x_n^*\rightarrow{0}$ a.e., $x_n\prec{x}$. Clearly, we may assume $x\neq{0}$. First, we claim that $x_n^*$ converges to zero globally in measure. Indeed, since for any $t>0$ and $n\in\mathbb{N}$,
\begin{equation}\label{equ:0:LLUKM=>KOC}
x_n^*(t)\leq{x_n^{**}(t)}\leq{x^{**}(t)},
\end{equation}
by Lemma 2.5 in \cite{CieKolPan} and Lemma \ref{lem:x**=0} for any $\epsilon>0$ there is $t_\epsilon>0$ such that for all $t\geq{t_\epsilon}$ and $n\in\mathbb{N}$ we have
\begin{equation}\label{equ:1:LLUKM=>KOC}
\max\{x_n^*(t),x^*(t)\}\leq\frac{\epsilon}{2}.
\end{equation}
Hence, since $x_n^*\rightarrow{0}$ locally in measure we get
\begin{align*}
\mu\left(t:x_n^*(t)\geq\epsilon\right)=\mu\left(t\in[0,t_\epsilon]:x_n^*(t)\geq\epsilon\right)\rightarrow{0},
\end{align*}
which proves our claim. Since $x^*\neq{0}$, there exists $t_0\in(0,t_\epsilon)$ such that $x^*(t_0)>0$. Define $\delta_0=x^*(t_0)/2$ and for any $n\in\mathbb{N}$,
\begin{equation*}
M_n=\left\{t\in[0,t_0]:x^*(t)\leq{x_n^*(t)}\right\}.
\end{equation*}
Then, by monotonicity of $x^*$ it is easy to see that for any $t\in[0,t_0]$,
\begin{equation*}
x^*(t)\geq{}x^*(t_0)>\delta_0.
\end{equation*}
Therefore, for any $n\in\mathbb{N}$ we obtain
\begin{equation*}
M_n\subset\left\{t:x_n^*(t)>\delta_0\right\},
\end{equation*}
whence, by the claim $x_n^*\rightarrow{0}$ globally in measure we get $\mu(M_n)\rightarrow{0}$. Moreover, by assumption that $x_n^*\rightarrow{0}$ a.e., we may suppose without loss of generality that $x_n^*(t_0)\rightarrow{0}.$ Clearly, for any $n\in\mathbb{N}$ we observe
\begin{equation*}
(x_n^*-x^*)^+\chi_{[t_0,t_\epsilon]}\leq{x_n^*}\chi_{[t_0,t_\epsilon]}\leq{x_n^*(t_0)}\chi_{[t_0,t_\epsilon]}.
\end{equation*}
Hence, by symmetry of $E$ it is easy to notice that
\begin{equation*}
\norm{(x_n^*-x^*)^+\chi_{[t_0,t_\epsilon]}}{}{}\leq{x_n^*(t_0)}\norm{\chi_{[t_0,t_\epsilon]}}{}{}
\end{equation*}
for each $n\in\mathbb{N}$. Consequently, since ${x_n^*(t_0)}$ converges to zero, this implies
\begin{equation}\label{equ::part:1:conv}
\norm{(x_n^*-x^*)^+\chi_{[t_0,t_\epsilon]}}{}{}\rightarrow{0}.
\end{equation}
Now, according to condition \eqref{equ:0:LLUKM=>KOC} we have
\begin{equation*}
(x_n^*-x^*)^+\leq{x_n^*}\prec{x^*}\quad\textnormal{and}\quad(x^*-x_n^*)^+\leq{x^*}
\end{equation*}
for every $n\in\mathbb{N}$. In consequence, by Hardy's lemma \cite{BS} it follows that 
\begin{equation*}
\int_{0}^{t}((x_n^*-x^*)^+)^*\chi_{[0,\mu(M_n)]}\leq\int_{0}^{t}x^*\chi_{[0,\mu(M_n)]}
\end{equation*}
for all $n\in\mathbb{N}$. Thus, by Proposition 1.1 \cite{ChSu} for any $t>0$ and $n\in\mathbb{N}$ we conclude
\begin{align*}
\left((x_n^*-x^*)^+\chi_{M_n}\right)^{**}(t)&=\frac{1}{t}\int_{0}^{t}\left((x_n^*-x^*)^+\chi_{M_n}\right)^*\\
&\leq\frac{1}{t}\int_{0}^{t}((x_n^*-x^*)^+)^*\chi_{[0,\mu(M_n)]}\\
&\leq\left(x^*\chi_{[0,\mu(M_n)]}\right)^{**}(t)\leq{x}^{**}(t).
\end{align*}
Hence, since $E$ is a symmetric space and 
$$(x_n^*-x^*)^+\chi_{M_n}=(x_n^*-x^*)^+\chi_{[0,t_0]}$$ 
for a.a. on $I$ and for any $n\in\mathbb{N}$, we have
\begin{equation*}
\norm{(x_n^*-x^*)^+\chi_{[0,t_0]}}{E}{}=\norm{(x_n^*-x^*)^+\chi_{M_n}}{E}{}\leq\norm{x^*\chi_{[0,\mu(M_n)]}}{E}{}
\end{equation*}
for each $n\in\mathbb{N}$. In consequence, since $x^*\chi_{[0,\mu(M_n)]}\leq{x^*}$ and $\mu(M_n)\rightarrow{0}$, by assumption that $x$ is a point of order continuity and by Lemma 2.6 \cite{CieKolPan} it follows 
\begin{equation*}
\norm{(x_n^*-x^*)^+\chi_{[0,t_0]}}{E}{}\rightarrow{0},
\end{equation*}
whence, by condition \eqref{equ::part:1:conv} and by the triangle inequality of the norm in $E$ we obtain
\begin{equation}\label{equ:6:LLUKM=>KOC}
\norm{(x_n^*-x^*)^+\chi_{[0,t_\epsilon]}}{E}{}\rightarrow{0}.
\end{equation}
Furthermore, since $\min\{x_n^*,x^*\}\leq{x^*}$ and $\min\{x_n^*,x^*\}\rightarrow{0}$ a.e., using analogous argument as previously we conclude 
\begin{equation*}
\norm{\min\{x_n^*,x^*\}}{E}{}\rightarrow{0}.
\end{equation*}
Therefore, since $x_n^*=\min\{x_n^*,x^*\}+(x_n^*-x^*)^+$ for any $n\in\mathbb{N}$, by condition \eqref{equ:6:LLUKM=>KOC} we get
\begin{equation}\label{equ:8:LLUKM=>KOC}
\norm{x_n^*\chi_{[0,t_\epsilon]}}{E}{}\rightarrow{0}.
\end{equation}
Now, we need to only show 
\begin{equation}\label{equ:second:part:conv}
\norm{x^*_n\chi_{[t_\epsilon,\infty)}}{E}{}\rightarrow{0}.
\end{equation}
Define for any $n\in\mathbb{N}$,
\begin{equation*}
y_n=\left(\frac{1}{n}\int_{0}^{n}x^*\right)\chi_{[0,n)}+x^*\chi_{[n,\infty)}.
\end{equation*}
It is obvious that $y_n=y_n^*$ for all $n\in\mathbb{N}$. We claim that $y_n\prec{x}$ for any $n\in\mathbb{N}$. In fact, by monotonicity of the maximal function of $x^*$, for any $n\in\mathbb{N}$ and $t\leq{n}$ we have
\begin{equation*}
y_n^{**}(t)=\frac{1}{t}\int_{0}^{t}\left(\frac{1}{n}\int_{0}^{n}x^*\right)\chi_{[0,n)}=\frac{1}{n}\int_{0}^{n}x^*\leq{x}^{**}(t).
\end{equation*}
Next, taking $n\in\mathbb{N}$ and $t>n$ by definition of $y_n$ we get immediately $y_n^{**}(t)=x^{**}(t)$ which proves our claim. Since $x^*(\infty)=0$, by Lemma \ref{lem:x**=0} it follows that $y_n^*\rightarrow{0}$ a.e. on $I$. Moreover, for each $n\in\mathbb{N}$ we have 
\begin{equation*}
\norm{y_n}{E}{}\leq{\phi(n)}x^{**}(n)+\norm{x^*\chi_{[n,\infty)}}{E}{}.
\end{equation*}
In consequence, since $x$ is a point of order continuity, by Lemma 2.6 in \cite{CieKolPan} and by assumption that $\phi(n)x^{**}(n)\rightarrow{0}$ as $n\rightarrow\infty$ it follows that  
\begin{equation}\label{equ:spec:seq:conv}
\norm{y_n}{E}{}\rightarrow{0}.
\end{equation}
Since $x_n^*\rightarrow{0}$ a.e., we may assume without loss of generality that $x_n^*(t_\epsilon)\rightarrow{0}$. Moreover, we may find a subsequence $(n_k)\subset\mathbb{N}$ such that for any $k\in\mathbb{N}$,
\begin{equation*}
x_{n_k}^*(t_\epsilon)\leq\frac{1}{k}\int_{0}^{k}x^*.
\end{equation*}
Fix $k\in\mathbb{N}$. Then, for any $t\in[0,k)$ we get
\begin{equation*}
\left(x_{n_k}^*\chi_{[t_\epsilon,\infty)}\right)^*(t)\leq{}x_{n_k}^*(t_\epsilon)\leq{y_k}^*(t).
\end{equation*}
Furthermore, since $y_k^{**}(t)=x^{**}(t)$ for all $t\geq{}k$ and by assumption $x_n\prec{x}$ for any $n\in\mathbb{N}$, considering $t\geq{}k$ we observe
\begin{equation*}
\left(x_{n_k}^*\chi_{[t_\epsilon,\infty)}\right)^{**}(t)\leq{}x_{n_k}^{**}(t)\leq{x^{**}(t)}={y_k}^{**}(t).
\end{equation*}
Hence, we obtain $x_{n_k}^*\chi_{[t_\epsilon,\infty)}\prec{y_k}$ for each $k\in\mathbb{N}$. In consequence, by symmetry of $E$ this yields that 
\begin{equation*}
\norm{x_{n_k}^*\chi_{[t_\epsilon,\infty)}}{E}{}\leq\norm{y_k}{E}{}
\end{equation*}  
for any $k\in\mathbb{N}$. Thus, by condition \eqref{equ:spec:seq:conv} and by Lemma 2.3 in \cite{ChDSS} we prove condition \eqref{equ:second:part:conv}. Finally, according to conditions \eqref{equ:8:LLUKM=>KOC} and \eqref{equ:second:part:conv} we get the end of the proof.
\end{proof}

The above theorem implies immediately the following result.

\begin{corollary}\label{coro:KOC}
		Let $E$ be a symmetric space on $I=[0,\alpha)$, where $\alpha=1$ or $\alpha=\infty$, with the fundamental function $\phi$. Then, the following conditions are equivalent.
	\begin{itemize}
		\item[$(i)$]  $E$ is order continuous and in case when $\alpha=\infty$, for any $x\in{E}$, $$\phi(s)x^{**}(s)\rightarrow{0}\quad\textnormal{as}\quad{}s\rightarrow\infty.$$
		\item[$(ii)$] $E$ is $K$-order continuous and in case when $\alpha =\infty$, for any $x\in{E}$, $$x^*(\infty)=0.$$
	\end{itemize} 
\end{corollary}

Now, we present an example of a nontrivial symmetric space that is $K$-order continuous.
 
\begin{example}
	 Assuming that $\phi$ is a quasiconcave function on $I$ with $\inf_{t>0}{\phi(t)}=0$ and $\sup_{t>0}{\frac{t}{\phi(t)}}=\infty$, by Theorem 1.3 in \cite{KamHJLee} it is well known that the subspace of all points of order continuity in the Marcinkiewicz space $M_\phi$ coincides with 
	\begin{equation*}
	(M_\phi)_a=\left\{y\in{M_\phi}:\lim\limits_{t\rightarrow{0^+,\infty}}\phi(t)y^{**}(t)=0\right\}\neq\{0\}.
	\end{equation*}
	Finally, according to Corollary \ref{coro:KOC} we observe that $(M_\phi)_a$ is $K$-order continuous.
\end{example}

In the next result we establish an essential correspondence between an $H_g$ point and a point of order continuity in symmetric spaces under an additional assumption.

\begin{proposition}\label{pro:Hg=>OC}
Let $E$ be a symmetric space and let $x\in{E}$, $x^*(\infty)=0$. If $x$ is an $H_g$ point, then $x$ is a point of order continuity in $E$.
\end{proposition}

\begin{proof}
Let $(x_n)\subset{E^+}$, $x_n\leq\abs{x}{}{}$ and $x_n\rightarrow{0}$ a.e. Since $x^*(\infty)=0$, by property 2.12 in \cite{KPS} it follows that $x_n^*(t)\rightarrow{0}$ for all $t>0$. Moreover, by Proposition 1.7 \cite{BS} we have $x_n^*\leq{x^*}$ a.e. and consequently it is easy to show that $x_n^*\rightarrow{0}$ in measure. Hence, we get
$0\leq{x^*-x_n^*\leq{x^*}}$ a.e. for any $n\in\mathbb{N}$ and ${x^*-x_n^*\rightarrow{x^*}}$ in measure. Therefore, it is obvious that $\norm{x^*-x_n^*}{E}{}\leq\norm{x^*}{E}{}$ for any $n\in\mathbb{N}$. Furthermore, by assumption that $x$ is an $H_g$ point, in view of Theorem 3.3 \cite{CieKolPlu} it follows that $x^*$ is an $H_g$ point. Then, by Theorem 3.8 \cite{CieKolPlu} and by the Fatou property we conclude
\begin{equation*}
\norm{x^*-x_n^*}{E}{}\rightarrow\norm{x^*}{E}{}.
\end{equation*}  
Thus, since $x^*$ is an $H_g$ point, by symmetry of $E$ we get $\norm{x_n}{E}{}\rightarrow{0}$.
\end{proof}

$\begin{array}{l}
\textnormal{\small Maciej CIESIELSKI}\\
\textnormal{\small Institute of Mathematics}\\
\textnormal{\small Pozna\'{n} University of Technology}\\ 
\textnormal{\small Piotrowo 3A, 60-965 Pozna\'{n}, Poland}\\ 
\textnormal{\small email: maciej.ciesielski@put.poznan.pl;}
\end{array}$

\end{document}